\documentclass[reqno,twoside]{amsart}
\usepackage{amsfonts}
\usepackage{amsmath,amssymb}
\usepackage{epic}
\usepackage{pstricks}
\usepackage{graphicx}
\usepackage{xcolor}

\usepackage{a4wide,amsmath,amssymb,latexsym,amsthm}
\usepackage{eucal}
\newcommand{\fl}[2]{(-\partial_x^{\,2})^#1#2}

\usepackage{graphicx}

\newtheorem{theorem}{Theorem}[section]
\newtheorem{proposition}[theorem]{Proposition}
\newtheorem{remark}[theorem]{Remark}
\newtheorem{lemma}[theorem]{Lemma}
 
\newtheorem{definition}[theorem]{Definition}

\numberwithin{equation}{section}

\graphicspath{{./Imgs/}}

\newcommand{\R}{\mathbb R}
 
\newcommand{\N}{\mathbb N}

\newcommand{\be}{\begin{equation}}
\newcommand{\ee}{\end{equation}}
\newcommand{\ba}{\begin{eqnarray}}
\newcommand{\ea}{\end{eqnarray}}
\newcommand{\beq}{\begin{equation}}
\newcommand{\eeq}{\end{equation}}

\usepackage{color}
\definecolor{red}{rgb}{0,0,0}

\usepackage[textsize=small]{todonotes}

\numberwithin{equation}{section}

\def\Omc{\R\setminus (-1,1)}

\def\RR{{\mathbb{R}}}
\def\NN{{\mathbb{N}}}
\def\N{{\mathbb{N}}}

\def\Om{(-1,1)}

\def\bOm{\overline{\Om}}

\usepackage{color}
\usepackage{stmaryrd}


\begin{document}

\title[Nonlocal heat equation]{Null controllability from the exterior of a one-dimensional nonlocal heat equation}

\author{Mahamadi Warma}
\address{M. Warma, George Mason University, Department of Mathematical Sciences,
Fairfax, VA 22030 (USA)}
\email{mwarma@gmu.edu, mjwarma@gmail.com }

\author{Sebasti\'an Zamorano}
\address{S. Zamorano, Universidad de Santiago de Chile, Departamento de
Matem\'atica y Ciencia de la Computaci\'on, Facultad de Ciencia, Casilla 307-Correo 2,
Santiago, Chile.}
 \email{sebastian.zamorano@usach.cl}

\thanks{The work of the first author is partially supported by the Air Force Office of Scientific Research (AFOSR) under Award NO [FA9550-18-1-0242] and the second author  is supported by the Fondecyt Postdoctoral Grant NO [3180322]}

\begin{abstract}
We consider the null controllability problem from the exterior for the one dimensional heat equation on the interval $(-1,1)$ associated with the fractional Laplace operator $(-\partial_x^2)^s$, where $0<s<1$. We show that there is a control function which is localized in a nonempty open set $\mathcal O\subset \left(\RR\setminus(-1,1)\right)$, that is, at the exterior of the interval $(-1,1)$, such that the system is null controllable at any time $T>0$ if and only if $\frac 12<s<1$.
\end{abstract}

\maketitle

\section{Introduction and main results}

In the present paper we consider a nonlocal version of the boundary controllability problem for the heat equation in the one dimensional case. The standard problem to find a boundary control for the heat equation is a well-known topic and has been studied by several authors. We refer for example to the pioneer works of MacCamy, Mize and  Seidman \cite{maccamy1968approximate, maccamy1969approximate} and the books of Zuazua \cite{Zua1} and Lions \cite{Lio1}, and the references therein, for a complete analysis and review on this topic.

Next, we describe our problem and state the main result. We consider  the fractional heat equation in the interval $(-1,1)$. That is,
\begin{align}\label{HE-EX}
\begin{cases}
\partial_t u + (-\partial_x^2)^{s} u = 0 & \mbox{ in }\; (-1,1)\times(0,T),\\
u=g\chi_{\mathcal{O}\times(0,T)} &\mbox{ in }\; (\Omc)\times (0,T), \\
u(\cdot,0) = u_0&\mbox{ in }\; (-1,1).
\end{cases}
\end{align}
In \eqref{HE-EX}, $u=u(x,t)$ is the state to be controlled,  $T>0$ and $0<s<1$ are real numbers, $g=g(x,t)$ is the exterior control function which is localized in a nonempty open subset $\mathcal O$ of ($\Omc)$ and
$(-\partial_x^2)^s$ denotes the fractional Laplace operator given formally for a smooth function $u$ by the following singular integral:
\begin{align*}
(-\partial_x^2)^su(x):=C_{s}\,\mbox{P.V.}\int_{\R}\frac{u(x)-u(y)}{|x-y|^{1+2s}}\;dy,\;\;x\in\R.
\end{align*}
We refer to Section \ref{sec-2} for the precise definition. We mention that it has been shown in \cite{War-ACE} that a boundary control (the case where the control $g$ is localized in a nonempty subset of the boundary) does not make sense for the fractional Laplace operator. By \cite{War-ACE}, for the fractional Laplace operator, the classical boundary control problem must be replaced by an exterior control problem. That is, the control function must be localized outside the open set $(-1,1)$ as it is formulated in \eqref{HE-EX}.

We shall show that for every $u_0\in L^2(-1,1)$ and  $g\in L^2((0,T);H^s(\Omc))$, the system \eqref{HE-EX} has a weak solution $u\in L^2((0,1);L^2(-1,1))$ (see Section \ref{sec-3}). In that case the set of reachable states is given by
\begin{align*}
\mathcal R(u_0,T)=\left\{u(\cdot, T):\; g\in L^2((0,T);H^s(\Omc))\right\}.
\end{align*}
We say that the system \eqref{HE-EX} is null controllable at time $T>0$, if $0\in \mathcal R(u_0,T)$. The system is said to be exact controllable at $T>0$, if $\mathcal R(u_0,T)=L^2(-1,1)$.
We say that the system \eqref{HE-EX} is controllable to the trajectories at $T>0$, if for any trajectory $\tilde{u}$ solution of \eqref{HE-EX} with initial datum $\tilde{u}_0\in L^2(-1,1)$ and without control $(g\equiv 0)$, and for every initial datum $u_0\in L^2(-1,1)$, there exists a control function $g\in L^2((-1,1);H^s(\Omc))$ such that the associated weak solution $u$ of \eqref{HE-EX} satisfies
\begin{align*}
u(\cdot,T)=\tilde{u}(\cdot,T),\quad \text{a.e. in }  (-1,1).
\end{align*}
The system is said to be approximately controllable at time $T>0$, if $\mathcal R(u_0,T)$ is dense in $L^2(-1,1)$.
We refer to Section \ref{sec-2} for the definition of the function spaces involved.

We mention that as in the classical local case ($s=1$) discussed in \cite[Chapter 2]{Zua1}, we have the following situation for the nonlocal case.
We observe that  solutions of \eqref{HE-EX} are of class $C^\infty$ far from $(\Omc)$ at time $t=T$. This shows that the elements of $\mathcal R(u_0,T)$ are $C^\infty$ functions in $(-1,1)$. Thus, the exact controllability may not hold. For this reason we shall study the null controllability of the system. However, since the system \eqref{HE-EX} is linear, the null controllability is equivalent to the controllability to trajectories.

The following theorem is the main result of the paper.

\begin{theorem}\label{T1}
Let $0<s<1$ and let $\mathcal O\subset(\Omc)$ be an arbitrary nonempty open set.
Then the following assertions hold.
\begin{enumerate}
\item If $\frac 12<s<1$, then the system \eqref{HE-EX} is null controllable at any time $T>0$. 

\item If $0<s\le \frac 12$, then the system \eqref{HE-EX} is not null controllable at time $T>0$.


\end{enumerate}
\end{theorem}

We mention that in the proof of Theorem \ref{T1}, we shall heavily exploit the fact that the eigenvalues $\{\lambda_n\}_{n\in\N}$ of the realization of $(-\partial_x^2)^s$ in $L^2(-1,1)$ with the zero Dirichlet exterior condition (see Section \ref{sec-2}) satisfy the following asymptotics (see e.g. \cite{kwasnicki2012eigenvalues}):
\begin{align}\label{lam}
\lambda_n=\left(\frac{n\pi}{2}-\frac{(2-2s)\pi}{8}\right)^{2s}+O\left(\frac{1}{n}\right)\;\text{ as }\, n\to\infty.
\end{align}

Recall that by Theorem \ref{T1}(b), the system \eqref{T1} is not null controllable at time $T>0$, if $0<s\le \frac 12$. It has been recently shown in \cite{War-ACE} that the system is indeed approximately controllable at any time $T>0$. The result obtained in \cite{War-ACE} is more general since it includes the $N$-dimensional case and the fractional diffusion equation, that is, the case where $\partial_tu$ is replaced by the Caputo time fractional derivative $\mathbb D_t^\alpha u$ of order $0<\alpha\le 1$. Of course, the case $\alpha=1$ includes \eqref{HE-EX}.

The  null controllability from the interior (that is, the case where the control function is localized in a nonempty subset $\omega$ of $(-1,1)$) of the one-dimensional fractional heat equation has been recently investigated in \cite{biccari2017controllability} where the authors have shown that the system is null controllable at any time $T>0$ if and only if $\frac 12<s<1$.  The interior null controllability of the Schr\"odinger and wave equations have been studied in \cite{Umb}. The approximate controllability from the exterior of the super-diffusive system, that is, the case where $u_{tt}$ is replaced by the Caputo time fractional derivative $\mathbb D_t^\alpha u$ of order $1<\alpha<2$, has been very recently considered in \cite{CLR-MW}. The case of the (possible) strong damping fractional wave equation has been investigated in \cite{WaZa}.

The fractional heat equation \eqref{HE-EX} defined in all the real line arises from a probabilistic process in which a particle makes long jumps random walks with a small probability, see for instance \cite{bucur2016nonlocal,Val}. Besides, this type of process occurs in real life phenomena quite often, see for example the biological observations in \cite{viswanathan1996levy} and the marine predators  in \cite{humphries2010environmental}.

Regarding the exterior control problem, in many real life applications a control is placed outside 
(disjoint from) the observation domain $\Omega$ where the PDE is satisfied. 
Some examples of control problems where this 
situation may arise are (but not limited to): 
 Acoustic testing, when the loudspeakers are placed far from the 
aerospace structures \cite{larkin1999direct}; 
 Magnetotellurics (MT), which is a technique to infer earth's subsurface
electrical conductivity from surface measurements \cite{unsworth2005new, CWeiss_BvBWaanders_HAntil_2018a}; 
 Magnetic drug targeting (MDT), where drugs with ferromagnetic particles 
in suspension are injected into the body and the external magnetic field is 
then used to steer the drug to relevant areas, for example, solid tumors \cite{HAntil_RHNochetto_PVenegas_2018a, HAntil_RHNochetto_PVenegas_2018b,lubbe1996clinical}; and Electroencephalography (EEG) is used to record electrical activities in brain \cite{niedermeyer2005electroencephalography,williams1974electroencephalography}, 
in case one accounts for the neurons disjoint from the brain, one will obtain
an external control problem. Besides, we mention that some preliminaries results about numerical analysis have been obtained  in the recent work \cite{AnKW}. 
%

The study of fractional order operators and nonlocal PDEs is nowadays a topic with interest to the mathematics and scientist communities due to the numerous applications that nonlocal PDEs provide. A motivation for this
growing interest relies in the large number of possible applications in the modeling of several complex phenomena for
which a local approach turns out to be inappropriate or limited. Indeed, there is an ample spectrum of situations in
which a nonlocal equation gives a significantly better description than a local PDE of the problem one wants to analyze.
Among others, we mention applications in turbulence, anomalous transport and diffusion, elasticity, image processing, porous media flow, wave propagation in heterogeneous high contrast media (see e.g. \cite{HAntil_SBartels_2017a, BBC,Val,SV2} and their references).
Also, it is well known that the fractional Laplace operator is the generator of the so-called $s$-stable L\'evy process, and it is often used in
stochastic models with applications, for instance, in mathematical finance 
(see e.g. \cite{AnKW,QDu_MGunzburger_RBLehoucq_KZhou_2013a}).
One of the main differences between nonlocal models and classical PDEs is that the fulfillment of a nonlocal equation at a point involves the values of the function far away from that point. We refer to \cite{BBC,Caf1, Caf3} and their references for more applications and information on this topic. 

The rest of the paper is structured as follows. In Section \ref{sec-2} we introduce the function spaces needed to study our problem and we give some intermediate known results that are needed in the proof of our main results. In Section \ref{sec-3} we show the well-posedness of the system \eqref{HE-EX} and its associated dual system and we give an explicit representation of solutions in terms of series for both problems. Finally, in Section \ref{prof-ma-re} we give the proof of our main results. 

\section{Preliminary results}\label{sec-2}
In this section we give some notations and recall some known results as they are needed in the proof of our main results.
We start with fractional order Sobolev spaces. 

For $0<s<1$ and $\Omega\subset\RR$ an arbitrary open set,  we let
\begin{align*}
H^{s}(\Omega):=\left\{u\in L^2(\Omega):\;\int_\Omega\int_\Omega\frac{|u(x)-u(y)|^2}{|x-y|^{1+2s}}\;dxdy<\infty\right\},
\end{align*}
and we endow it with the norm defined by
\begin{align*}
\|u\|_{H^{s}(\Omega)}:=\left(\int_\Omega|u(x)|^2\;dx+\int_\Omega\int_\Omega\frac{|u(x)-u(y)|^2}{|x-y|^{1+2s}}\;dxdy\right)^{\frac 12}.
\end{align*}
We set
\begin{align*}
\widetilde H_0^{s}(\Omega):=\Big\{u\in H^{s}(\R):\;u=0\;\mbox{ a.e. in }\;\R\setminus \Omega\Big\}
\end{align*}

We shall denote by $\widetilde  H^{-s}(\Omega)$ the dual space of $\widetilde  H_0^{s}(\Omega)$, that is, $\widetilde H^{-s}(\Omega)=(\widetilde H_0^{s}(\Omega))^\star$.

For more information on fractional order Sobolev spaces, we refer to \cite{NPV,War} and their references.

Next, we give a rigorous definition of the fractional Laplace operator. To do this, we need the following function space:
\begin{align*}
\mathcal L_s^{1}(\R):=\left\{u:\R\to\R\;\mbox{ measurable and }\; \int_{\R}\frac{|u(x)|}{(1+|x|)^{1+2s}}\;dx<\infty\right\}.
\end{align*}
For $u\in \mathcal L_s^{1}(\R)$ and $\varepsilon>0$ we set
\begin{align*}
(-\partial_x^2)_\varepsilon^s u(x):= C_{s}\int_{\{y\in\R:\;|x-y|>\varepsilon\}}\frac{u(x)-u(y)}{|x-y|^{1+2s}}\;dy,\;\;x\in\R,
\end{align*}
where $C_{s}$ is a normalization constant given by
\begin{align}\label{CNs}
C_{s}:=\frac{s2^{2s}\Gamma\left(\frac{2s+1}{2}\right)}{\pi^{\frac
12}\Gamma(1-s)}.
\end{align}
The {\em fractional Laplacian}  $(-\partial_x^2)^s$ is defined for $u\in \mathcal L_s^{1}(\R)$ by the following singular integral:
\begin{align}\label{fl_def}
(-\partial_x^2)^su(x):=C_{s}\,\mbox{P.V.}\int_{\R}\frac{u(x)-u(y)}{|x-y|^{1+2s}}\;dy=\lim_{\varepsilon\downarrow 0}(-\partial_x^2)_\varepsilon^s u(x),\;\;x\in\R,
\end{align}
provided that the limit exists. 
We notice that if $u\in \mathcal L_s^{1}(\R)$, then $ v:=(-\partial_x^2)_\varepsilon^s u$ exists for every $\varepsilon>0$, $v$ being also continuous at the continuity points of  $u$.  
For more details on the fractional Laplace operator we refer to \cite{Caf3,NPV,GW-CPDE,War} and their references.

Next, we consider the realization of $(-\partial_x^2)^s$ in $L^2(-1,1)$ with the zero Dirichlet exterior condition $u=0$ in $\Omc$. More precisely, we consider the closed and bilinear form $\mathcal F:\widetilde  H_0^{s}\Om\times \widetilde  H_0^{s}\Om\to\RR$ given by
\begin{align*}
\mathcal F(u,v):=\frac{C_{s}}{2}\int_{\R}\int_{\R}\frac{(u(x)-u(y))(v(x)-v(y))}{|x-y|^{1+2s}}\;dxdy,\;\;u,v\in \widetilde  H_0^{s}\Om.
\end{align*}
Let $(-\partial_x^2)_D^s$ be the selfadjoint operator on $L^2(-1,1)$ associated with $\mathcal F$ in the sense that
\begin{equation*}
\begin{cases}
D((-\partial_x^2)_D^s)=\Big\{u\in \widetilde H_0^{s}\Om,\;\exists\;f\in L^2(-1,1),\;\mathcal F(u,v)=(f,v)_{L^2(-1,1)}\;\forall\;v\in \widetilde H_0^{s}\Om\Big\},\\
(-\partial_x^2)_D^su=f.
\end{cases}
\end{equation*}
It is easy to see that
\begin{equation}\label{DLO}
D((-\partial_x^2)_D^s)=\left\{u\in \widetilde H_0^{s}\Om:\; (-\partial_x^2)^su\in L^2(-1,1)\right\},\;\;\;
(-\partial_x^2)_D^su=((-\partial_x^2)^su)|_{(-1,1)}.
\end{equation}
Then $(-\partial_x^2)_D^s$ is the realization of $(-\partial_x^2)^s$ in $L^2(-1,1)$ with the condition $u=0$ in $\Omc$. It is well-know (see e.g. \cite{CL-WA}) that the operator $-(-\partial_x^2)_D^s$ generates a strongly continuous submarkovian semigroup $(e^{-t(-\partial_x^2)_D^s})_{t\ge 0}$ on $L^2(-1,1)$.
It has been shown in \cite{SV2} that $(-\partial_x^2)_D^s$ has a compact resolvent and its eigenvalues form a non-decreasing sequence of real numbers $0<\lambda_1\le\lambda_2\le\cdots\le\lambda_n\le\cdots$ satisfying $\lim_{n\to\infty}\lambda_n=\infty$.  In addition, if $\frac 12\le s<1$, then the eigenvalues are of finite multiplicity.
Let $\{\varphi_n\}_{n\in\NN}$ be the orthonormal basis of eigenfunctions associated with the eigenvalues $\{\lambda_n\}_{n\in\NN}$. Then $\varphi_n\in D((-\partial_x^2)_D^s)$ for every $n\in\NN$,  $\{\varphi_n\}_{n\in\NN}$ is total in $L^2(-1,1)$ and satisfies 
\begin{equation}\label{ei-val-pro}
\begin{cases}
(-\partial_x^2)^s\varphi_n=\lambda_n\varphi_n\;\;&\mbox{ in }\;(-1,1),\\
\varphi_n=0\;&\mbox{ in }\;\Omc.
\end{cases}
\end{equation}

Next, for $u\in H^{s}(\R)$ we introduce the {\em nonlocal normal derivative $\mathcal N_s$} given by 
\begin{align}\label{NLND}
\mathcal N_su(x):=C_{s}\int_{-1}^1\frac{u(x)-u(y)}{|x-y|^{1+2s}}\;dy,\;\;\;x\in\R \setminus\bOm,
\end{align}
where $C_{s}$ is the constant given in \eqref{CNs}.
We notice that since equality is to be understood a.e., we have that \eqref{NLND} is the same as for a.e. $x\in\Omc$.
By \cite[Lemma 3.2]{GSU},  for every $u\in H^{s}(\R)$, we have that $\mathcal N_su\in H_{\rm loc}^s(\Omc)$. 
We mention that the operator $\mathcal N_s$ has been called "interaction operator" in  \cite{AnKW,QDu_MGunzburger_RBLehoucq_KZhou_2013a}. Several properties of $\mathcal N_s$ have been studied in \cite{CL-WA,DRV}. 

The following unique continuation property which shall play an important role in the proof of Theorem \ref{T1} has been recently obtained in \cite[Theorem 16]{War-ACE}.

\begin{lemma}\label{lem-UCD}
Let $\lambda>0$ be a real number  and $\mathcal O\subset(\Omc)$ an arbitrary  nonempty open set. 
If $\varphi\in D((-\partial_x^2)_D^s)$ satisfies
\begin{equation*}
(-\partial_x^2)_D^s\varphi=\lambda\varphi\;\mbox{ in }\;(-1,1)\;\mbox{ and }\; \mathcal N_s\varphi=0\;\mbox{ in }\;\mathcal O,
\end{equation*}
then $\varphi=0$ in $\R$. 
\end{lemma}

For more details on the Dirichlet problem associated with the fractional Laplace operator we refer the interested reader to \cite{BWZ1,Grub,RS2-2,RS-DP,War-ACE} and their references.

We conclude this section with the following integration by parts formula.

\begin{lemma}
Let $u\in \widetilde H_0^s\Om$ be such that $(-\partial_x^2)^s u\in L^2(-1,1)$ and $\mathcal N_su\in L^2(\Omc)$. Then for every $v\in H^s(\RR)$, the following identity
\begin{align}\label{Int-Part}
\frac{C_{s}}{2}\int_{\RR}\int_{\RR}\frac{(u(x)-u(y))(v(x)-v(y))}{|x-y|^{1+2s}}\;dxdy=\int_{-1}^1v(x)(-\partial_x^2)^su(x)\;dx+\int_{\Omc}v(x)\mathcal N_su(x)\;dx,
\end{align}
holds.
\end{lemma}

We refer to \cite[Lemma 3.3]{DRV} (see also \cite[Proposition 3.7]{War-ACE}) for the proof and more details.

\section{Well-posedness of the parabolic problem}\label{sec-3}

This section is devoted to the well-posedness and the explicit representation in terms of series for solutions to the system \eqref{HE-EX} and its associated dual system. 

Throughout the remainder of the article, $\{\varphi_n\}_{n\in\NN}$ denotes the orthonormal basis of eigenfunctions of the operator $(-\partial_x^2)_D^s$ associated with the eigenvalues $\{\lambda_n\}_{n\in\NN}$, and $(e^{-t(-\partial_x^2)_D^s})_{t\ge 0}$ denotes the strongly continuous semigroup on $L^2(-1,1)$ generated by the operator $-(-\partial_x^2)_D^s$. 

Furthermore, for a given measurable set $E\subseteq \RR^N$ ($N\ge 1)$, we shall denote by $(\cdot,\cdot)_{L^2(E)}$ the scalar product in $L^2(E)$ and   by $\mathcal{D}(E)$ we mean the space of all continuously infinitely differentiable functions with compact support in $E$.  For a given $u\in L^2(-1,1)$ and $n\in\NN$, we shall let $u_n:=(u,\varphi_n)_{L^2(-1,1)}$.
Finally, given a Banach space $\mathbb X$ and its dual $\mathbb X^\star$, we shall denote by $\langle \cdot,\cdot\rangle_{\mathbb X^\star,\mathbb X}$ (simply $\langle \cdot,\cdot\rangle$ if there is no confusion) they duality pairing.

\subsection{Representation of solutions to the system \eqref{HE-EX}}

Let $T>0$ be a fixed real number, $u_0\in L^2(-1,1)$, $g\in L^2((0,T);H^s(\Omc))$ and consider the following two systems:
\begin{align}\label{SS1}
\begin{cases}
\partial_t v + (-\partial_x^2)^{s} v = 0 & \mbox{ in }\; (-1,1)\times(0,T),\\
v=0 &\mbox{ in }\; (\Omc)\times (0,T), \\
v(\cdot,0) = u_0&\mbox{ in }\; (-1,1),
\end{cases}
\end{align}
and \begin{align}\label{SS2}
\begin{cases}
\partial_t w + (-\partial_x^2)^{s} w = 0 & \mbox{ in }\; (-1,1)\times(0,T),\\
w=g &\mbox{ in }\; (\Omc)\times (0,T), \\
w(\cdot,0) = 0&\mbox{ in }\; (-1,1).
\end{cases}
\end{align}

Notice that the system \eqref{SS1} can be rewritten as the following Cauchy problem:
\begin{align*}
\begin{cases}
\partial_t v + (-\partial_x^2)_D^{s} v = 0 & \mbox{ in }\; (-1,1)\times(0,T),\\
v(\cdot,0) = u_0&\mbox{ in }\; (-1,1).
\end{cases}
\end{align*}
Hence, using semigroups theory and the spectral theorem for selfadjoint operators, one has the following result.
\begin{proposition}
For every $u_0\in L^2(-1,1)$, there is a unique function
$$v\in C([0,T];L^2(-1,1))\cap L^2((0,T); \widetilde H_0^s\Om\cap H^1((0,T);\widetilde H^{-s}\Om)$$ 
satisfying \eqref{SS1} and is given for a.e. $x\in (-1,1)$ and every $t\in [0,T]$ by
\begin{align}\label{eq-SS1}
v(x,t)= e^{-t(\partial_x^2)_D^s}u_0(x)=\sum_{n=1}^{\infty}u_{0,n}e^{-\lambda_n t}\varphi_n(x).
\end{align}
\end{proposition}

Next, we consider the system \eqref{SS2}.
\begin{definition}\label{DF1}
Let  $g\in L^2((0,T);H^s(\Omc))$. By a weak solution of \eqref{SS2}, we mean a function $w\in L^2((0,T);H^s(\RR))$ such that $w=g$  a.e. in $(\Omc)\times (0,T)$ and the identity
\begin{align*}
\int_0^T\langle -\partial_t\phi +(-\partial_x^2)^s\phi, w\rangle\;dt=\int_{-1}^1w(x,T)\phi(x,T)\;dx+\int_0^T\int_{\Omc}g\mathcal N_s\phi\;dxdt
\end{align*}
holds, for every function $\phi\in  C([0,T];L^2(-1,1))\cap L^2((0,T); \widetilde H_0^s\Om)\cap H^1((0,T);\widetilde H^{-s}\Om)$ with $\mathcal N_s\phi\in L^2((0,T);L^2(\Omc))$.
\end{definition}

We have the following existence result which proof is inspired from the local case contained in the monograph \cite[pp. 180-185]{lasiecka2000control}.
\begin{proposition}
For every  $g\in L^2((0,T);H^s(\Omc))$, the system \eqref{SS2} has a weak solution $w\in L^2((0,T);H^s(\RR))$ given by
\begin{align}\label{eq-SS2}
w(x,t)=\sum_{n=1}^{\infty}\left(\int_0^t (g(\cdot,\tau), \mathcal{N}_s\varphi_n)_{L^2(\Omc)}  e^{-\lambda_n(t-\tau)}d\tau\right)\varphi_n(x).
\end{align}
\end{proposition}

\begin{proof}
Recall that the operator,
$$(-\partial_x^2)_D^s:D((-\partial_x^2)_D^s)\to L^2(-1,1), \;\;u\mapsto(-\partial_x^2)_D^su:=(-\partial_x^2)^su\;\mbox{ in } (-1,1),$$
 defined in \eqref{DLO} is a selfadjoint operator on $L^2(-1,1)$. We denote by  $\left(D((-\partial_x^2)_D^s)\right)^\star$ the dual space of $D((-\partial_x^2)_D^s)$ with respect to the pivot space $L^2(-1,1)$ so that $D((-\partial_x^2)_D^s)\hookrightarrow L^2(-1,1)\hookrightarrow \left(D((-\partial_x^2)_D^s)\right)^\star$.
  
Let $\mathbb D$ be the nonlocal Dirichlet map given by
\begin{align}\label{D}
\mathbb Dg=u \Longleftrightarrow  (-\partial_x^2)^su=0 \mbox{ in }\; (-1,1)\;\mbox{ and } u=g\mbox{ in }\Omc.
\end{align}
It is well-known (see e.g. \cite{AnKW,War-ACE}) that for every $g\in H^s(\Omc)$, there is a unique function $u\in H^s(\RR)$ satisfying \eqref{D}.

Next, let the operator $\mathbb B$ be given by
\begin{align}\label{B}
\mathbb B: H^s(\Omc)\to \left(D((-\partial_x^2)_D^s)\right)^\star,\; g\mapsto \mathbb Bg:=-(-\partial_x^2)_D^s\mathbb Dg.
\end{align}

Firstly, we claim that for every $u\in D((-\partial_x^2)_D^s)$ and $g\in H^s(\Omc)$ we have 
\begin{align}\label{CLa}
\int_{-1}^1u\mathbb Bg\;dx=\int_{\Omc}g\mathcal N_su\;dx.
\end{align}
Indeed, let $u\in D((-\partial_x^2)_D^s)$ and $g\in H^s(\Omc)$. Applying the integration by parts formula \eqref{Int-Part} and using \eqref{D}-\eqref{B}, we get that
\begin{align}\label{bb}
\int_{-1}^1u\mathbb Bg\;dx=&-\int_{-1}^1\mathbb Dg(-\partial_x^2)^su=-\int_{-1}^1u(-\partial_x^2)^s\mathbb Dg\;dx+\int_{\Omc}\mathbb Dg\mathcal N_su\;dx-\int_{\Omc}u\mathcal N_s(\mathbb Dg)\;dx\notag\\
=&\int_{\Omc}\mathbb Dg\mathcal N_su\;dx=\int_{\Omc}g\mathcal N_su\;dx,
\end{align}
where we have also used the facts that $u=0$ in $\Omc$ (since $u\in D((\partial_x^2)_D^s)\subset \widetilde H_0^s\Om$) and $(\mathbb Dg)|_{\Omc}=g$ by \eqref{D}. We have shown the claim \eqref{CLa}.

Secondly, with the above setting, proceeding as in the local case (see \cite[pp. 180-185]{lasiecka2000control} and the references therein), using semigroups theory, \eqref{bb} and the spectral theorem for selfadjoint operators, we can deduce that for every function $g\in L^2((0,T);H^s(\Omc))$,  there exists a function $w\in L^2((0,T);H^s(\RR))$ which is a weak solution of \eqref{SS2} and is given by
\begin{align*}
w(x,t)=&\int_0^te^{-(t-\tau)(-\partial_x^2)_D^s}(\mathbb Bg)(x,\tau)\;d\tau\\
=&\sum_{n=1}^{\infty}\left(\int_0^t ((\mathbb Bg)(\cdot,\tau), \varphi_n)_{L^2(-1,1)}  e^{-\lambda_n(t-\tau)}d\tau\right)\varphi_n(x)\\
=&\sum_{n=1}^{\infty}\left(\int_0^t (g(\cdot,\tau), \mathcal N_s\varphi_n)_{L^2(\Omc)}  e^{-\lambda_n(t-\tau)}d\tau\right)\varphi_n(x).
\end{align*}
We have shown \eqref{eq-SS2} and the proof is finished.
\end{proof}

We have the following existence and explicit representation in terms of series of solutions to  \eqref{HE-EX}.

\begin{theorem}\label{expl-2}
Let $T>0$. Then for every $u_0\in L^2(-1,1)$ and $g\in L^2((0,T);H^s(\Omc))$, the system \eqref{HE-EX} has a weak solution $u\in L^2((0,T);L^2(-1,1))$ given by
\begin{align}\label{4}
u(x,t)=\sum_{n=1}^{\infty}u_{0,n}e^{-\lambda_n t}\varphi_n(x)+\sum_{n=1}^{\infty}\left(\int_0^t (g(\cdot,\tau), \mathcal{N}_s\varphi_n)_{L^2(\mathcal O)}  e^{-\lambda_n(t-\tau)}d\tau\right)\varphi_n(x).
\end{align}
\end{theorem}

\begin{proof}
Let $u_0\in L^2(-1,1)$ and $g\in L^2((0,T);H^s(\Omc))$.
Let $v\in C([0,T];L^2(-1,1))$ be the weak solution of \eqref{SS1} and $w\in L^2((0,T);H^s(\RR))$ the weak solution of \eqref{SS2} with $g$ replaced by $g\chi_{\mathcal{O}\times(0,T)}$.  Set $u:=u+v$. Then it is clear that $u\in L^2((0,T);L^2(-1,1))$ and is a weak solution of \eqref{HE-EX}. The representation \eqref{4}  follows directly from \eqref{eq-SS1} and \eqref{eq-SS2}. The proof is finished.
\end{proof}

We conclude this section with the following remark.
\begin{remark}
{\em We make the following observations.
\begin{enumerate}
\item   Theorem \ref{expl-2} is the fractional version of the classical local  heat equation with inhomogeneous boundary data, and it is the so-called \emph{boundary control semigroup formula}. We refer for instance to the book of Lasiecka and Triggiani \cite{lasiecka2000control} and the paper of Fattorini \cite{fattorini1968boundary} for more details on the local case.
\item The representation \eqref{4} of solutions to the system \eqref{HE-EX} is contained  in  \cite{War-ACE} for the case where $\partial_t$ is replaced with the Caputo time-fractional derivative $\mathbb D_t^\alpha$ ($0<\alpha\le 1$), and for a smooth function $g\in\mathcal D((0,T)\times \Omc)$. In that case, since the function $g$ is smooth, one has that the solution $u\in C([0,T];L^2(-1,1))$. This is not the case here, since $g\in L^2((0,T);H^s(\Omc))$ and is not smooth enough. However, proceeding as in \cite[pp 180-185]{lasiecka2000control} and the references therein, the time regularity of the solution $u$ can be improved. Since this is not the goal of the present paper, and since such a result and the representation \eqref{4} are not needed in the proof of our main results, we will not go into details.
\end{enumerate}
}
\end{remark}

\subsection{Representation of solutions to the associated dual system}\label{sds}

Using the classical integration by parts formula, we have that the following backward system,
\begin{equation}\label{Dual}
\begin{cases}
-\partial_t \psi +(-\partial_x^2)^s\psi=0\;\;&\mbox{ in }\; (-1,1)\times (0,T),\\
\psi=0&\mbox{ in }\;(\Omc)\times (0,T),\\
\psi(\cdot,T)=\psi_0&\mbox{ in }\;(-1,1),
\end{cases}
\end{equation} 
can be viewed as the dual system associated with \eqref{HE-EX}.

We have the following existence result.
\begin{theorem}\label{theo-48}
Let $T>0$ be a real number and  $\psi_0\in  L^2(-1,1)$. Then the system \eqref{Dual} has a unique weak solution $\psi\in C([0,T];L^2(-1,1))$ given for a.e. $x\in (-1,1)$ and every $t\in [0,T]$ by
\begin{align}\label{eq-25}
\psi(x,t)=\sum_{n=1}^{\infty}\psi_{0,n}e^{-\lambda_n(T-t)}\varphi_{n}(x).
\end{align}
In addition the following assertions hold.
\begin{enumerate}
\item There is a constant $C>0$ such that for all $t\in [0,T]$,
\begin{equation}\label{Dual-EST-1}
 \|\psi(\cdot,t)\|_{L^2(-1,1)}\le C\|\psi_0\|_{L^2(-1,1)}.
\end{equation}
\item For every $t\in[0,T)$ fixed, $\mathcal{N}_s \psi(\cdot,t)$ exists, belongs to $L^2(\Omc)$ and is given by
\begin{align}\label{norm-der}
\mathcal{N}_s \psi(x,t)=\sum_{n=1}^{\infty}\psi_{0,n}e^{-\lambda_n(T-t)}\mathcal{N}_s\varphi_{n}(x),
\end{align}
where we recall that $\psi_{0,n}:=(\psi_0,\varphi_n)_{L^2(-1,1)}$.
\end{enumerate}
\end{theorem}

\begin{proof}
 Using the spectral theorem for selfadjoint operators with compact resolvent, we are reduced to look for a solution $\psi$ of the form
\begin{align*}
\psi(x,t)=\sum_{n=1}^\infty(\psi(\cdot,t),\varphi_n)_{L^2(-1,1)}\varphi_n(x).
\end{align*}
Replacing this expression in \eqref{Dual} and letting $\psi_n(t):=(\psi(\cdot,t),\varphi_n)_{L^2(-1,1)}$, we get that $\psi_n(t)$ solves the following ordinary differential equation:
\begin{align*}
-\psi_n^{\prime}(t)+\lambda_n\psi_n(t)=0, \;t\in (0,T);\;\;\mbox{ and }\;\psi_n(T)=\psi_{0,n}.
\end{align*}
It is straightforward to show that $\psi$ is given by \eqref{eq-25}. Noticing that $\psi(x,t)=e^{-(T-t)(-\partial_x^2)_D^s}\psi_0(x)$ (where we recall that $(e^{-t(-\partial_x^2)_D^s})_{t\ge0 }$ is the strongly continuous semigroup on $L^2(-1,1)$ generated by the operator $-(\partial_x^2)_D^s$), and using semigroups theory, it is well-known that $\psi\in C([0,T];L^2(-1,1))$.
The estimate \eqref{Dual-EST-1} and the identity \eqref{norm-der} can also be easily justified. The proof is finished.
\end{proof}

We conclude this section with the following remark.
\begin{remark}
Using semigroups theory, it is well-known that the solution $\psi\in C([0,T];L^2(-1,1))$ of the backward system \eqref{Dual} enjoys the following regularity:
$$\psi\in C([0,T];L^2(-1,1))\cap L^2((0,T); \widetilde H_0^s\Om)\cap H^1((0,T);\widetilde H^{-s}\Om).$$ 
\end{remark}

\section{Proof of the main results}\label{prof-ma-re}

In this section we give the proof of the main results of this work, namely Theorem \ref{T1}. In order to do this, we need first to establish some auxiliaries results that will be used in the proof.

\begin{lemma}\label{lem-1}
The system \eqref{HE-EX} is null controllable at time $T>0$ if and only if for every initial datum $u_0\in  L^2(-1,1)$, there exists a control function $g\in L^2((0,T);\widetilde H_0^s(\mathcal O))$ such that the solution $\psi$ of the dual system \eqref{Dual} satisfies
\begin{align}\label{NSC-NC}
\int_{-1}^1u_0(x)\psi(x,0)\;dx=\int_0^T\int_{\mathcal{O}} g(x,t)\mathcal{N}_s\psi(x,t)\;dxdt,
\end{align}
for every $\psi_0\in L^2(-1,1)$.
\end{lemma}

\begin{proof}
Let $u_0\in L^2(-1,1)$ and $g\in L^2((0,T);\widetilde H_0^s(\mathcal O))$. We write the solution $u$ of \eqref{HE-EX} as $u:=v+w$ where $v$ and $w$ are the solutions of \eqref{SS1}  and \eqref{SS2}, respectively.
Let $\psi$ be the solution of the dual problem \eqref{Dual}. Taking $\psi$ as a test function in Definition \ref{DF1} of a weak solution to the system \eqref{SS2}, using the integration by parts formula \eqref{Int-Part},  noticing that 
$$-\psi_t+(-\partial_x^2)^s\psi=0 \mbox{ in }(-1,1)\times (0,T),$$
and that $\psi=0$ in $(\Omc)\times (0,T)$,
we obtain that
\begin{align}\label{abcd}
0=&\int_0^T\langle v_t(\cdot,t)+(-\partial_x^2)^sv(\cdot,t),\psi(\cdot,t)\rangle\;dt+\int_0^T\langle-\psi_t(\cdot,t)+(-\partial_x^2)^s\psi(\cdot,t),w(\cdot,t)\rangle\;dt\notag\\
=&\int_{-1}^1\Big(v(x,T)\psi(x,T)-v(x,0)\psi(x,0)\Big)\;dx
+\int_0^T\int_{\Omc}\Big(v(x,t)\mathcal N_s\psi(x,t)-\psi(x,t)\mathcal N_sv(x,t)\Big)\;dxdt\notag\\
&+\int_{-1}^1w(x,T)\psi(x,T)\;dx+\int_0^T\int_{\Omc}w(x,t)\mathcal N_s\psi(x,t)\;dxdt\notag\\
=&-\int_{-1}^1v(x,0)\psi(x,0)\;dx +\int_{-1}^1\Big(v(x,T)+w(x,T)\Big)\psi(x,T)\;dx\notag\\
&+ \int_0^T\int_{\Omc}\Big(v(x,t)+w(x,t)\Big)\mathcal N_s\psi(x,t)\;dxdt.
\end{align}
Since $v(x,0)=u(x,0)=u_0(x)$ for a.e. $x\in (-1,1)$, and $u(x,t)=v(x,t)+w(x,t)$ for a.e. $(x,t)\in (-1,1)\times (0,T]$, and $u=g\chi_{\mathcal{O}\times(0,T)}$ in  $(\Omc)\times (0,T)$, it follows from \eqref{abcd} that 
\begin{align}\label{Ob-In}
\int_{-1}^1u(x,0)\psi(x,0)\;dx -\int_{-1}^1u(x,T)\psi(x,T)\;dx=\int_0^T\int_{\mathcal O}g(x,t)\mathcal N_s\psi(x,t)\;dxdt.
\end{align}

Now if \eqref{NSC-NC} holds, then it follows from \eqref{Ob-In} that $\displaystyle\int_{-1}^1u(x,T)\psi(x,T)\;dx=0$. Thus, we can deduce that $u(\cdot,T)=0$  in $(-1,1)$  and the system \eqref{HE-EX} is null controllable.

Conversely, if the system \eqref{HE-EX} is null controllable, that is, if $u(\cdot,T)=0$  in $(-1,1)$, then \eqref{NSC-NC} follows from \eqref{Ob-In} and the proof is finished.
\end{proof}

Next, using classical duality arguments, we can establish the following criterion for the null controllability.
\begin{lemma}\label{lem-2}
Let $\mathcal{O}\subset (\Omc)$ be an arbitrary nonempty open set. Then the following assertions are equivalent.
\begin{enumerate}
\item The system \eqref{HE-EX} is null controllable at time $T>0$ and  there is a constant $C>0$ such that 
\begin{align}\label{n1}
\|g\|_{L^2((0,T);\widetilde H_0^s(\mathcal O))}\leq C\|u_0\|_{L^2(-1,1)}.
\end{align}
\item For every $T>0$ and $\psi_0\in L^2(-1,1)$, let $\psi$ be the unique weak solution of the dual system \eqref{Dual} with final datum $\psi_0$. Then, there is a constant $C=C(T)>0$ such that the following observability inequality holds:
\begin{align}\label{n2}
\|\psi(\cdot,0)\|_{L^2(-1,1)}^2\leq C\int_0^T\int_{\mathcal{O}}|\mathcal{N}_s \psi(x,t)|^2dxdt.
\end{align}
\end{enumerate}
\end{lemma}

\begin{proof}
(b) $ \Rightarrow$ (a): We start by proving that the observability inequality \eqref{n2} implies the null controllability. 
Indeed, consider the linear subspace $\mathbb H$ of $L^2((0,T);\widetilde H^{-s}(\mathcal{O}))$ given by
\begin{align*}
\mathbb H:=\Big\{\mathcal{N}_s\psi\Big|_{\mathcal{O}\times(0,T)}:\; \psi \text{ solves the system }\eqref{Dual} \text{ with }\psi_0\in L^2(-1,1)\Big\}.
\end{align*}
Next, we consider the linear functional $F:\mathbb H\to\RR$ defined by
\begin{align*}
F(\mathcal{N}_s\psi):= \int_{-1}^1u_0(x)\psi(x,0)dx,
\end{align*}
where $u_0\in L^2(-1,1)$. It follows from the observability inequality \eqref{n2} that $F$ is well defined and bounded on $\mathbb H$. More precisely, there is a constant $C>0$ such that
\begin{align*}
|F(\mathcal{N}_s\psi)|\leq C \|u_0\|_{L^2(-1,1)}\|\mathcal{N}_s\psi\|_{L^2((0,T);\widetilde H^{-s}(\mathcal{O}))}.
\end{align*}
It follows from the Hahn--Banach theorem that $F$ can be extended to a bounded linear functional $\widetilde{F}:L^2((0,T);\widetilde H^{-s}(\mathcal{O}))\to \RR$ such that
\begin{align}\label{new1}
|\widetilde{F}v|\leq C\|u_0\|_{L^2(-1,1)}\|v\|_{L^2((0,T);\widetilde H^{-s}(\mathcal{O}))},\quad \forall v\in L^2((0,T);\widetilde H^{-s}(\mathcal{O})).
\end{align}
By the Riesz representation theorem, there is a $g\in (L^2((0,T);\widetilde H^{-s}(\mathcal{O})))^{*}=L^2((0,T);\widetilde H_0^s(\mathcal{O}))$ such that 
\begin{align*}
\widetilde{F}(\eta)=\int_0^T\langle \eta(\cdot,t),g(\cdot,t)\rangle\; dt, \quad \forall \;\eta\in L^2((0,T);\widetilde H^{-s}(\mathcal{O})).
\end{align*}
Moreover, we have that
\begin{align*}
\|\widetilde{F}\|=\|g\|_{L^2((0,T);\widetilde H_0^s(\mathcal{O}))}.
\end{align*}
Thus, we can deduce from \eqref{new1} that
\begin{align*}
\|g\|_{L^2((0,T);\widetilde H_0^s(\mathcal{O}))}\leq C\|u_0\|_{L^2(-1,1)}.
\end{align*}
Notice that $\mathcal{N}_s\psi \in L^2((0,T);L^2(\mathcal{O})) \subset L^2((0,T);\widetilde H^{-s}(\mathcal{O}))$. Therefore, using the definition of $F$ we can deduce that 
\begin{align*}
F(\mathcal{N}_s\psi)=\int_{-1}^1u_0(x)\psi(x,0)dx=\int_0^T\int_{\mathcal{O}}g(x,t)\mathcal{N}_s\psi(x,t)dxdt,
\end{align*}
for every $\psi_0\in L^2(-1,1)$. We have shown that there exists a control function $g\in L^2((0,T);\widetilde H_0^s(\mathcal{O}))$ satisfying \eqref{n1} and   
\begin{align*}
\int_0^T\int_{\mathcal{O}}g(x,t)\mathcal{N}_s\psi(x,t)dxdt=\int_{-1}^1u_0(x)\psi(x,0)dx,
\end{align*}
for every $\psi_0\in L^2(-1,1)$. Thus, it follows from Lemma \ref{lem-1} that the system  \eqref{HE-EX} is null controllable.\\

(a) $ \Rightarrow$ (b): Now, we show that the null controllability implies the observability inequality \eqref{n2}. Recall that from Lemma \ref{lem-1}, we have that for every $u_0\in L^2(-1,1)$, there exists a control $g \in L^2((0,T);\widetilde H_0^s(\mathcal O))$ such that the unique solution $\psi$ of the dual system \eqref{Dual} satisfies
\begin{align*}
\int_{-1}^1u_0(x)\psi(x,0)\;dx=\int_0^T\int_{\mathcal{O}} g(x,t)\mathcal{N}_s\psi(x,t)\;dxdt,
\end{align*}
for every $\psi_0\in L^2(-1,1)$.
Taking $u_0(x)=\psi(x,0)$ in the preceding identity, using \eqref{n1} and Young's inequality, we get that
\begin{align*}
\int_{-1}^1|\psi(x,0)|^2dx\leq \frac{C}{2\varepsilon}\|u_0\|_{L^2(-1,1)}^2+\frac{\varepsilon}{2}\int_0^T\int_{\mathcal O}|\mathcal{N}_s\psi(x,t)|^2dxdt,
\end{align*}
for every $\varepsilon>0$. Taking $\varepsilon=C$ and since $u_0(x)=\psi(x,0)$, we obtain \eqref{n2}. The proof is finished. 
\end{proof}

Finally, for the proof of Theorem \ref{T1} we also need the following technical result.

\begin{lemma}\label{uni-bound}
Let $\{\varphi_{k}\}_{k\in\N}$ be the orthogonal basis of normalized eigenfunctions of the operator $(-\partial_x^2)_D^s$ associated with the eigenvalues $\{\lambda_k\}_{k\in\N}$. Then, for every nonempty open set $\mathcal{O}\subset (\Omc)$, there is a scalar $\eta>0$ (independent of $k$)  such that for every $k\in\N$,
the function $\mathcal{N}_s\varphi_k$ is uniformly bounded  from below by $\eta$ in $L^2(\mathcal{O})$. Namely,
\begin{align}
\exists\; \eta>0\;\mbox{ such that }\; \|\mathcal{N}_s\varphi_k\|_{L^2(\mathcal{O})}\geq \eta,\;\;\; \forall \;k\in\N.
\end{align}
\end{lemma}

\begin{proof}
We prove the result in several steps.

{\bf Step 1}:
Firstly,  since $\varphi_k=0$  in $\Omc$ for every $k\in\NN$, it follows from the definition of the operators $(-\partial_x^2)^s$ and $\mathcal N_s$ that for almost every $x\in\mathcal O\subseteq(\Omc)$, we have 
\begin{align}\label{eq-eg}
(-\partial_x^2)^s\varphi_k(x)=C_{s}\mbox{P.V.}\int_{\RR}\frac{\varphi_k(x)-\varphi_k(y)}{|x-y|^{1+2s}}\;dy=C_{s}\int_{-1}^1\frac{\varphi_k(x)-\varphi_k(y)}{|x-y|^{1+2s}}\;dy=
\mathcal N_s\varphi_k(x).
\end{align}
We have shown that $(\mathcal N_s \varphi_k)|_{\mathcal O}=((-\partial_x^2)^s\varphi_k)|_{\mathcal O}$ for every $k\in\NN$. We notice that \eqref{eq-eg} holds not only for $\varphi_k$, but for all functions in $\widetilde H_0^s\Om$.

Secondly, let us introduce the auxiliary function $q:\RR\to[0,\infty)$ defined by: 
\begin{align}\label{q_def}
	q(x) := \begin{cases}
		0 & x\in \left(-\infty,-\frac 13\right),
		\\[7pt]
		\displaystyle \frac 92 \left(x+\frac 13\right)^2 & x\in \left(-\frac 13,0\right),
		\\[10pt]
		\displaystyle 1-\frac 92 \left(x-\frac 13\right)^2 & x\in \left(0,\frac 13\right),
		\\[7pt]
		1 & x\in \left(\frac 13,+\infty\right).
	\end{cases}
\end{align}
For any $\alpha>0$, we define the function $F_\alpha:\RR\to\RR$ as follows:
\begin{align*} 
F_\alpha(x) = F(\alpha x):= \sin\left(\alpha x+\frac{(1-s)\pi}{4}\right)-G(\alpha x),
\end{align*}
where $G$ is the Laplace transform of the function
\begin{align*}
\gamma(y):=\frac{\sqrt{4s}\sin(s\pi)}{2\pi}\frac{y^{2s}}{1+y^{4s}-2y^{2s}\cos(s\pi)}\exp\left(\frac 1\pi\int_0^{+\infty}\frac{1}{1+r^2}\log\left(\frac{1-r^{2s}y^{2s}}{1-r^2y^2}\right)\,dr\right).
\end{align*}
Next, we define the sequence of real numbers
\begin{align*} 
\mu_k:=\frac{k\pi}{2}-\frac{(1-s)\pi}{4},\quad k\geq 1.
\end{align*}
It has been shown in \cite[Example 6.1]{kwasnicki2010spectral} that $F_{\mu_k}$ is the solution of the system
\begin{align*}
\begin{cases}
\fl{s}{F_{\mu_k}(x)}=\mu_k F_{\mu_k}(x) & x> 0,
\\
F_{\mu_k}(x)=0 & x\le 0.
\end{cases}
\end{align*}
In other words, $\{F_{\mu_k}\}_{k\geq 1}$ are the eigenfunctions of  $(-\partial_x^2)^s$ on the interval $(0,\infty)$  with the zero Dirichlet exterior condition, and $\{\mu_k\}_{k\geq 1}$ are the corresponding eigenvalues.
Let us now define 
\begin{align*} 
	\varrho_k(x):=q(-x)F_{\mu_k}(1+x)+(-1)^k F_{\mu_k}(1-x),\quad x\in\RR,\quad k\geq 1.
\end{align*}
Notice that $F_{\mu_k}(1+x)=0$ for $x\leq -1$ and $F_{\mu_k}(1-x)=0$ for $x\geq 1$. This fact, together with the definition \eqref{q_def} of the function $q$ imply that, for all $k\geq 1$, $\varrho_k(x) = 0$ for $x\in\RR\setminus (-1,1)$. In addition, it follows from \cite[Lemma 1]{kwasnicki2012eigenvalues} that $\{\varrho_k\}_{k\in\NN}\subset D((-\partial_x^2)_D^s)$ and  there is a constant $C>0$ such that
\begin{align*}
\left| \fl{s}{\varrho_k}(x)-\mu_k^{2s}\varrho_k(x)\right|\leq \frac{C(1-s)}{\sqrt{2s}}\mu_k^{-1},\;\;\textrm{ for all }\;\; x\in (-1,1),\;k\geq 1.
\end{align*}
Furthermore, by \cite[Proposition 1]{kwasnicki2012eigenvalues}, there is a constant $C>0$ such that for every $k\ge 1$, we have
\begin{equation}
\|\varrho_k-\varphi_k\|_{L^2(-1,1)}\le 
\begin{cases}
\frac{C(1-s)}{k}\;\qquad&\mbox{ when }\; \frac 12\le s< 1,\\\
 \frac{C(1-s)}{k^{2s}}\;&\mbox{ when }\; 0<s<\frac 12.
 \end{cases}
\end{equation}

{\bf Step 2:}
Now, let $\mathcal{O}\subset \Omc$ be an arbitrary nonempty open set and assume that for every $\eta>0$ there exists $k\in\N$ such that 
\begin{align}\label{contr}
\|\mathcal{N}_s\varphi_k\|_{L^2(\mathcal{O})}<\eta.
\end{align}
It follows from \eqref{contr} that there is a subsequence $\{\varphi_{k_n}\}_{n\in\NN}$ such that
\begin{align}\label{contr-2}
\|\mathcal{N}_s\varphi_{k_n}\|_{L^2(\mathcal{O})}<\frac 1n,
\end{align}
for $n$ large enough. Since $L^2(\mathcal O)\hookrightarrow \widetilde H^{-s}(\mathcal O)$, it follows from \eqref{contr-2} that there is a constant $C>0$ (independent of $n$) such that for $n$ large enough, we have
\begin{align}\label{contr-3}
\|\mathcal{N}_s\varphi_{k_n}\|_{\widetilde H^{-s}(\mathcal{O})}\le \frac Cn.
\end{align}

{\bf Step 3}: Using the triangle inequality, we get that there is a constant $C>0$ such that
\begin{align}\label{nnn}
\|\varrho_{k_n}-\varphi_{k_n}\|_{\widetilde H_0^s\Om}^2\leq &C\|(-\partial_x^2)^s\varrho_{k_n}-(-\partial_x^2)^s\varphi_{k_n}\|_{L^2(-1,1)}^2\notag\\
\leq &C \Big(\|(-\partial_x^2)^s\varrho_{k_n}-\mu_{k_n}^{2s}\varrho_{k_n}\|_{L^2(-1,1)}^2+\|\varrho_{k_n}(\mu_{k_n}^{2s}-\lambda_{k_n})\|_{L^2(-1,1)}^2\notag \\
&+\|\lambda_{k_n}\varrho_{k_n}-(-\partial_x^2)^s\varphi_{k_n}\|_{L^2(-1,1)}^2\Big).
\end{align}
It follows from \eqref{nnn} and  Step 1 that there is a constant $C_{k_n}(s)>0$ which converges to zero as $n\to\infty$, such that 
\begin{align*} 
\|\varrho_{k_n}-\varphi_{k_n}\|_{\widetilde H_0^s(-1,1)}^2 \leq C_{k_n}(s).
\end{align*}
Let the operator $L$ be defined by 
$$L:  \widetilde H_0^{s}\Om\to \widetilde H^{-s}(\mathcal O),\; v\mapsto Lv:=((-\partial_x^2)^sv)|_{\mathcal O}=(\mathcal N_sv)|_{\mathcal O},$$
where we recall that $\widetilde H^{-s}(\mathcal O)=(\widetilde H_0^s(\mathcal O))^\star$.
 It has been shown in \cite[Lemma 2.2]{GRSU} that the operator $L$ is compact, injective with dense range. Let $B_1:=\overline{B}\left(\varrho_{k_n},C_{k_n}(s)\right)$ be the closed ball in $\widetilde H_0^s\Om$ with center in $\varrho_{k_n}$ and radius $C_{k_n}(s)$. Since $L$ is a compact operator, we have that the image of $B_1$, namely $L(B_1)$, is totally bounded in $\widetilde H^{-s}(\mathcal{O})$. Therefore, for every $\varepsilon>0$ there exists $N\in\N$ and $\{\psi_1,\ldots,\psi_N\}\subseteq B_1$ such that
\begin{align*}
L(B_1)\subseteq \bigcup_{j=1}^{N}\overline{B}_{\widetilde H^{-s}(\mathcal O)}(L(\psi_j),\varepsilon).
\end{align*}
We observe that,  the eigenfunction $\varphi_{k_n}$ belongs to $B_1$. Thus, there exists $j\in\{1,\ldots,N\}$ such that
\begin{align*}
L(\varphi_{k_n})\in \overline{B}_{\widetilde H^{-s}(\mathcal O)}(L(\psi_j),\varepsilon).
\end{align*}
We have shown that for $n$ large enough,
\begin{align*}
\|L(\varphi_{k_n})-L(\psi_j)\|_{\widetilde H^{-s}(\mathcal{O})}\leq \varepsilon.
\end{align*}
Since $\psi_j\in B_1$, firstly we obtain that $\varphi_{k_n}\rightarrow \psi_j$, as $n\to\infty$ in $\widetilde H_0^s\Om$  and, secondly we have that $\psi_j$ is an element of the spectrum $\{(\varphi_k,\lambda_k)\}_{k\geq 1}$. That is, $\psi_j$ is a solution of \eqref{ei-val-pro}. Finally, as $L(\varphi_{k_n})$ converges to zero in $\widetilde H^{-s}(\mathcal{O})$ (by \eqref{contr-3}), we can deduce that $L(\psi_j)=\mathcal N_s\psi_j=(-\partial_x^2)^s\psi_j=0$ a.e. in $\mathcal O$. It follows from the unique continuation property (Lemma \ref{lem-UCD})  that $\psi_j=0$  a.e. in $\RR$, which is a contradiction. The proof of the lemma is finished.
\end{proof}

Now we are ready to give the proof of our main results.

\begin{proof}[{\bf Proof of Theorem \ref{T1}}]
Let $u$ be the weak solution of \eqref{HE-EX} and $\psi$ the weak solution of the dual problem \eqref{Dual}.
Recall that by Lemma \ref{lem-1}, the system \eqref{HE-EX} is null controllable if and only if the identity \eqref{NSC-NC} holds. Moreover, from Lemma \ref{lem-2}, \eqref{NSC-NC}  is equivalent to the observability inequality \eqref{n2} for the dual system, that is, there exists  a constant $C>0$ such that
\begin{align}\label{obs-ine}
\|\psi(\cdot,0)\|_{L^2(-1,1)}^2\leq C\int_0^T\int_{\mathcal{O}}\left|\mathcal{N}_s \psi(x,t)\right|^2\;dxdt.
\end{align}
From Section \ref{sds}, the solution $\psi$ of \eqref{Dual} is given by
\begin{align*} 
\psi(\cdot,t)=\sum_{n=1}^{\infty}\psi_{0,n}e^{-\lambda_n(T-t)}\varphi_n(x),
\end{align*}
and its nonlocal normal derivative $\mathcal N_s\psi$ is given by
\begin{align*} 
\mathcal{N}_s\psi(\cdot,t)=\sum_{n=1}^{\infty}\psi_{0,n}e^{-\lambda_n(T-t)}\mathcal{N}_s\varphi_n(x).
\end{align*}
Therefore, using the above representations of $\psi$ and $\mathcal{N}_s\psi$,  the orthonormality of the eigenfunctions in $L^2(-1,1)$ and making the change of variable $T-t\to t$, we can deduce that the observability inequality \eqref{obs-ine} is equivalent to the following inequality:
\begin{align}\label{obs-ine-1}
\sum_{n=1}^{\infty}|\psi_{0,n}|^2 e^{-2\lambda_n T}\leq C\int_0^T\int_{\mathcal{O}}\left|\sum_{n=1}^{\infty}\psi_{0,n}e^{-\lambda_n t} \mathcal{N}_s\varphi_n(x)\right|^2dxdt.
\end{align}
By mean of the classical moment method (see e.g. \cite[Sections 2 and 3]{fattorini1971exact}), inequalities of the form \eqref{obs-ine-1} are well-known to be true if and only if the eigenvalues $\{\lambda_n\}_{n\in\NN}$ and eigenfunctions $\{\varphi_n\}_{n\in\NN}$ satisfy the following M\"{u}ntz condition:
\begin{align}\label{Lam}
\sum_{n=1}^\infty\frac{1}{\lambda_n}<\infty,
\end{align}
and 
\begin{align}\label{NOR}
\|\mathcal{N}_s\varphi_n\|_{L^2(\mathcal{O})}\ge \eta>0,\;\forall\;n\in\NN,
\end{align}
where the constant $\eta$ is independent of $n$.

Lemma \ref{uni-bound} implies that \eqref{NOR} holds. 
As we have mentioned in the introduction, the eigenvalues $\{\lambda_n\}_{n\ge 1}$ satisfy  \eqref{lam}. That is,
\begin{align}\label{lam1}
\lambda_n=\left(\frac{n\pi}{2}-\frac{(2-2s)\pi}{8}\right)^{2s}+O\left(\frac{1}{n}\right)\;\text{ as }\, n\to\infty.
\end{align}
Therefore, we easily see from \eqref{lam1} that the condition \eqref{Lam} is satisfied if and only if $\frac 12<s<1$. Instead, if $0<s\le \frac 12$, then the series diverges since it will have the behavior of the harmonic series. In conclusion, the observability inequality  \eqref{obs-ine} holds true when $\frac 12<s<1$, and it is false when $0<s\le \frac 12$. The proof of the theorem is finished.
\end{proof}

\noindent
{\bf Acknowledgment}: The authors would like to thank all the referees for their careful reading of the manuscript and their precise comments that have helped to improve the quality of the paper.

\bibliographystyle{abbrv}
\bibliography{biblio}

\end{document}